\newfont{\Bbb}{msbm10 scaled\magstephalf}
\documentclass[12pt,reqno]{amsart}
\usepackage{amsfonts}
\usepackage{amsmath, amsthm, amscd, amsfonts}
\usepackage{amssymb}
\usepackage{amsmath}
\usepackage{xcolor}
\allowdisplaybreaks[4]
 \newtheorem{thm}{Theorem}[section]
 \newtheorem{cor}[thm]{Corollary}
 
 \newtheorem{prop}[thm]{Proposition}
 \theoremstyle{definition}
 \newtheorem{defn}[thm]{Definition}
\theoremstyle{remark}
 \newtheorem{rem}[thm]{Remark}
 \newtheorem{exm}[thm]{Example}
 \numberwithin{equation}{section}
\newcommand{\pf}{\begin{proof}}
\newcommand{\zb}{\end{proof}}
\newcommand{\la}{\langle}
\newcommand{\ra}{\rangle}
\newcommand{\Ker}{\mathop{\rm Ker}\nolimits}
\begin{document}
\title[perturbations of Toeplitz operators]{Representing kernels of perturbations of Toeplitz operators by backward shift-invariant subspaces}
\author[Y. Liang ]{Yuxia Liang}
\address{Yuxia Liang \newline School of Mathematical Sciences,
Tianjin  Normal University,  Tianjin 300387, P.R. China.} \email{liangyx1986@126.com}
\author[J. R. Partington]{Jonathan R. Partington}
\address{Jonathan R. Partington \newline School of Mathematics,
  University of Leeds, Leeds LS2 9JT, United Kingdom.}
 \email{J.R.Partington@leeds.ac.uk}
\begin{abstract}
 It is well known the kernel of a Toeplitz operator is nearly invariant under the backward shift $S^*$. This paper shows that ker\-nels of finite rank perturbations of Toeplitz operators are near\-ly $S^*$-invariant with finite defect. This enables us to apply a recent theorem by Chalendar--Gallardo--Partington to represent the kernel in terms of backward shift-invariant subspaces, which we identify in several important cases.
\end{abstract}
\maketitle

\section{Introduction}
Let $H(\mathbb{D})$ be the space of all analytic functions on the open unit disc $\mathbb{D}$. The Hardy space $H^2:=H^2(\mathbb{D})$  is defined by $$H^2=\{f\in H(\mathbb{D}):\;f(z)=\sum_{n=0}^\infty a_nz^n\;\mbox{with}\;\|f\|^2:=\sum_{n=0}^\infty |a_n|^2<+\infty \}.$$

The limit $\lim\limits_{r\rightarrow 1^{-}} f(re^{it})$ exists almost everywhere, which gives the values of $f$ on the unit circle $\mathbb{T}$. Since the $H^2$ norm of $f$ and the $L^2(\mathbb{T})$ norm of
its boundary function coincide, $H^2$ embeds isometrically as a closed subspace of $L^2(\mathbb{T})$ via $$\sum_{n=0}^\infty a_n z^n\mapsto \sum_{n=0}^\infty a_n e^{int}.$$ This indicates a natural orthogonal decomposition $L^2(\mathbb{T})=H^2\oplus \overline{H_0^2},$  where $H^2$ is identified with the subspace spanned by $\{e^{int}:\;n\geq 0\}$ and $\overline{H_0^2}$ is the subspace spanned by $\{e^{int}:\;n< 0\},$ respectively.

Let $L^\infty:=L^\infty (\mathbb{T})$ be the space containing all essentially bounded functions on $\mathbb{T}$. And $H^\infty:=H^\infty(\mathbb{D})$ is the Banach algebra of bounded analytic functions on $\mathbb{D}$ with the norm defined $$\|f\|_\infty=\sup\limits_{z\in \mathbb{D}}|f(z)|.$$
Similarly, the radial boundary function of an $H^\infty$ function belongs to $L^\infty,$ and then $H^\infty$ can be viewed as a Banach subalgebra of $L^\infty.$

We recall an inner function is an $H^\infty$ function that has unit modulus almost everywhere on $\mathbb{T}$. An outer function is a function $f\in H^1$ which can be written in the form
$$f(re^{i\eta})=\alpha \exp(\frac{1}{2\pi} \int_0^{2\pi} \frac{e^{it}+re^{i\eta}}{e^{it}-re^{i\eta}}k(e^{it})dt)$$ for $re^{i\eta}\in \mathbb{D}$, where $k$ is a real-valued integrable function and $|\alpha|=1.$
It  is known that each $f\in H^1\setminus\{0\}$ has a factorization $f=\theta\cdot u$, where $\theta$ is inner and $u$ is outer. This factorization is unique up to a constant of modulus $1$ (cf. \cite{CGP2}).

  Let $P: \;L^2(\mathbb{T})\rightarrow H^2$ be the orthogonal projection on $H^2$ defined by a Cauchy integral $$(Pf)(z)=\int \frac{f(\zeta)}{1-\overline{\zeta}z}dm(\zeta),\;|z|<1.$$

Given $g\in L^\infty$, the Toeplitz operator $T_g:\;H^2\rightarrow H^2$ is defined by $$T_g f=P(gf)$$ for any $f\in H^2.$  If $\theta$ is an inner function, then $\Ker T_{\overline{\theta}}$ is the model space $K_{\theta}=H^2\ominus \theta H^2=H^2\cap \theta \overline{H_0^2}$ (cf. \cite{GMR,Ha}). It has also been proved that $\|T_g\|=\|g\|_\infty$ and $T_g^*=T_{\overline{g}}$ (cf. \cite{BH}). For more investigations into Toeplitz operators, the reader can refer to \cite{CP,CMP,Sa2} and so on.\vspace{1mm}

Beurling's theorem states that the subspaces $\theta H^2$ with inner function $\theta$ constitute the nontrivial invariant subspaces for the unilateral shift $S: \;H^2\rightarrow  H^2$ defined by $[S f](z) = z f(z).$ Also the model space $K_\theta$ is invariant under the backward shift $S^*:\;H^2\rightarrow  H^2$ (cf. \cite[Proposition 5.2]{GMR}) defined by  $$S^* f(z)=\frac{f(z)-f(0)}{z}\;\;(f\in H^2,\;z\in \mathbb{D}).$$

The invariant subspace problem is still an unresolved problem in operator theory and there are various related investigations (cf. \cite{CaP,CaP1}). Moreover, the study of nearly $S^*$-invariant subspaces has attracted a lot of attention (cf. \cite{hitt,Sa1,CaP1}).

\begin{defn} A subspace $M\subset H^2$ is called nearly $S^*$-invariant if $S^* f\in M$ whenever $f\in M$ and $f(0)=0.$ Furthermore, a subspace $M\subset H^2$ is said to be nearly $S^*$-invariant with defect $m$ if there is an $m$-dimensional subspace $F$ such that $S^* f\in M + F$ whenever $f\in M$ with $f(0)=0$; we call $F$ the defect space.\end{defn}

If $f\in \Ker T_g$ with $f(0)=0,$ so $gf\in \overline{H_0^2}$ and then $g(\overline{z}f)\in \overline{H_0^2}$. Since $\overline{z}f\in H^2$, this implies $S^*f=\overline{z}f\in \Ker T_g$, which shows the kernel of a Toeplitz operator is nearly $S^*$-invariant. 
Motivated by this well-known result, we continue to examine a question which has a close link with the invariant subspace problem:\vspace{1mm}

\emph{Given a Toeplitz operator $T_g$ acting on Hardy space $H^2,$ is the kernel of a rank $n$ perturbation of $T_g$ nearly $S^*$-invariant with finite defect?}\vspace{1mm}

We recall that an operator $T:\;\mathcal{H}\rightarrow \mathcal{H}$ of rank $n$ on a Hilbert space $\mathcal{H}$  takes the form \begin{eqnarray*} Th =\sum_{i=1}^n \la h, u_i \ra v_i\;\mbox{for all}\; h\in \mathcal{H},\end{eqnarray*} where $\{u_i\}$ and $\{v_i\}$ are orthogonal sets in $\mathcal{H}$ (we may also suppose that
$\{u_i\}$ is orthonormal).  For simplicity, write $A_n:=\{1, 2,\cdots, n\}$ and let $|\Lambda|$ stand for the number of integers in a set $\Lambda.$

A rank $n$ perturbation of the Toeplitz operator $T_g: H^2\rightarrow H^2$ denoted by $R_n: H^2\rightarrow H^2$ is defined by \begin{eqnarray}R_n(h)=T_gh+Th= T_gh+\sum_{i=1}^n \la h, u_i\ra v_i\label{RnT}\end{eqnarray} with orthonormal set $\{u_i\}$ and orthogonal set $\{v_i\}$ in $H^2$.\vspace{1mm}

The rest of the paper is organized as follows. In Section 2, we discuss the nearly $S^*$-invariant subspace $\Ker R_n$ with finite defect for several important classes of symbols and present the corresponding defect space in each case. Then we apply a recent theorem by Chalendar--Gallardo--Partington to represent the kernel of the operator $R_1$ in terms of backward shift-invariant subspaces in Section 3. The challenging task here is to identify the subspaces in question, which we do in various important cases. Note that even in the nearly $S^*$-invariant (defect $0$) case, this is known to be a difficult question in general.

\section{nearly $S^*$-invariant $\Ker R_n$ with finite defect}
In this section, we prove that the kernel of the operator $R_n$ in \eqref{RnT} is nearly $S^*$-invariant with finite defect for various important cases, especially identify the finite-dimensional defect spaces. First of all, we recall a useful theorem for later use.
\begin{thm}\cite[Theorem 4.22]{GMR}  \label{thm infinity}For $\psi,\;\varphi \in L^\infty$, the operator $T_\psi T_{\varphi}$ is a Toeplitz operator if and only if either $\overline{\psi}\in H^\infty$ or $\varphi\in H^\infty.$ In both cases, $T_\psi T_\varphi=T_{\psi\varphi}.$ \end{thm}

So for all $g\in L^\infty,$ it holds that
\begin{eqnarray} T_{\overline{z}}T_g=T_{\overline{z}g}=T_{g\overline{z}}.\label{overlinez}\end{eqnarray}

For every $h\in \Ker R_n$, it follows that\begin{eqnarray}
  T_g h +\sum_{i=1}^n\la h,u_i\ra v_i=0. \label{Tg1}\end{eqnarray}
Letting $S^*=T_{\overline{z}}$ act on both sides of \eqref{Tg1} and using \eqref{overlinez}, we have $$T_{g\overline{z}}h+\sum_{i=1}^n\la h,u_i\ra S^* v_i=0.$$
Now let $h\in \Ker R_n$ satisfy $h(0)=0$, and then the above equation implies the following equivalent expressions.
\begin{eqnarray}
 &&T_g(\frac{h}{z})+ \sum_{i=1}^n\la h,u_i\ra S^* v_i=0 \label{TgS}\\&\Leftrightarrow& g\frac{h}{z} +\sum_{i=1}^n\la h,u_i\ra  S^*v_i \in \overline{H_0^2}. \label{S*}
 \end{eqnarray}

So the question of nearly $S^*$-invariant $\Ker R_n$ with finite defect is that: \emph{for each $h\in \Ker R_n$ with $h(0)=0$, find a vector $w$ in some suitable finite-dimensional space $F$ such that} \begin{eqnarray*}S^* h+w=\frac{h}{z}+w \in \Ker R_n,\end{eqnarray*}
which is equivalent to the following equations.
 \begin{eqnarray}
 && T_g(\frac{h}{z}+w)+\sum_{i=1}^n\la \frac{h}{z}+w, u_i\ra v_i=0
\label{question}\\ &\Leftrightarrow&  g(\frac{h}{z}+w)+\sum_{i=1}^n\la \frac{h}{z}+w, u_i\ra v_i\in \overline{H_0^2}. \label{questionw} \end{eqnarray}  Next we will construct the defect space $F$ in several important cases.

\subsection{$g=0$ a.e. on $\mathbb{T}$}
In this case, $R_n$ is a rank-$n$ operator and  equation \eqref{question} with $g=0$ implies
\begin{eqnarray*} \Ker R_n=\bigcap_{i=1}^n(\bigvee\{ u_i\})^{\bot} =H^2\ominus(\bigvee \{u_i, i\in A_n\}), \end{eqnarray*}
where $\bigvee$ denotes the closed linear span in $H^2.$

For any $h\in \Ker R_n$ with $h(0)=0,$ it always holds that
\begin{eqnarray*}  S^*h\in \Ker R_n\oplus(\bigvee \{u_i, i\in A_n\}) =H^2,  \end{eqnarray*} which 
gives the following elementary observation on the nearly $S^*$-invariant subspace $\Ker R_n$ with finite defect.
\begin{prop}\label{prop g=0} Suppose $g=0$ almost everywhere on $\mathbb{T}.$ Then the subspace $\Ker R_n$ is nearly $S^*$-invariant with defect $n$ and defect space $$F=\bigvee \{ u_i, \;i\in A_n\}.$$ \end{prop}

\subsection{$g=\theta$ an inner function}
In this case $T_\theta f=\theta f$ is an isometric multiplication operator on $H^2$. For each $h\in \Ker R_n$ with $h(0)=0,$ the relation \eqref{S*} becomes
\begin{eqnarray} \theta \frac{h}{z} +\sum_{i=1}^n \la h, u_i\ra S^*v_i=0.\label{theta2}\end{eqnarray}
The required relation \eqref{questionw} turns into
\begin{eqnarray*} \theta (\frac{h}{z}+w) +\sum_{i=1}^n  \la \frac{h}{z}+w, u_i\ra v_i=0.   \end{eqnarray*}
Combining it with  \eqref{theta2}, the above equation is equivalent to
\begin{eqnarray}(\theta w-\sum_{k=1}^n \la h, u_k\ra S^*v_k) +\sum_{i=1}^n \la(\theta w-\sum_{k=1}^n \la h, u_k\ra S^*v_k), \theta u_i\ra v_i =0. \quad \label{theta3} \end{eqnarray}
Now choosing $$w= \overline{\theta} (\sum_{k=1}^n \la h, u_k\ra  S^*v_k)=\sum_{k=1}^n \la h, u_k\ra  T_{\overline{\theta}}(S^*v_k)\in H^2,$$ the required equation \eqref{theta3} holds. So we can obtain a theorem on the nearly $S^*$-invariant $\Ker R_n$ with finite defect.

\begin{thm}\label{thm theta1} Suppose $g=\theta$ an inner function. Then  the subspace $\Ker R_n$ is nearly $S^*$-invariant with defect at most $n$ and defect space $$F=\bigvee \{T_{\overline{\theta}}(S^*v_i),\; i\in A_n\}.$$ \end{thm}

\begin{exm}\label{remm} For  $g(z)=z^m$  ($m\in \mathbb{N}$), $\Ker R_n$ is nearly $S^*$-invariant with defect at most $n$ and defect space $F=\bigvee \{(S^*)^{m+1} (v_i),\; i\in A_n\}$. \end{exm}

\subsection{$g=f_1\overline{f_2}$ with $f_j \in \mathcal{G}H^\infty$ for $j=1,2$.} Here $\mathcal{G}H^\infty$ denotes the set of all invertible elements in $H^\infty.$ In \cite{Bour}, Bourgain   proved: \emph{If $g$ is a bounded measurable function on $\mathbb{T},$ then the condition $\int_\pi \log|g| dm >-\infty$ $(m$ is the normalized invariant measure on $\mathbb{T})$ is the necessary and sufficient condition for $g\neq 0$ to be of the form $g=f_1\cdot \overline{f_2}$ where $f_1, f_2\in H^\infty$.} The interested reader can also refer to \cite[Theorem 4.1]{Ba} for a matricial version with norm estimates. In this subsection, we suppose $f_j \in \mathcal{G}H^\infty$ for $j=1,2$, and then Theorem \ref{thm infinity} ensures that $T_{f_1\overline{f_2}}=T_{\overline{f_2}}T_{f_1}.$ \vspace{0.1mm}

For each $h\in \Ker R_n$ with $h(0)=0,$ \eqref{TgS} can be rewritten as
\begin{eqnarray} T_{\overline{f_2}}T_{f_1}(\frac{h}{z}) +\sum_{i=1}^n\la h,u_i\ra S^*v_i=0, \label{f*12}\end{eqnarray}
which together with Theorem \ref{thm infinity} imply \begin{eqnarray} \frac{h}{z} +\sum_{i=1}^n\la h,u_i\ra T_{f_1^{-1}}T_{\overline{f_2}^{-1}}(S^*v_i)=0. \label{Tf*12}\end{eqnarray}
The required equation \eqref{question} is changed into \begin{eqnarray*}T_{\overline{f_2}}T_{f_1}(\frac{h}{z}+w) +\sum_{i=1}^n\la \frac{h}{z}+w,u_i\ra v_i=0, \end{eqnarray*}
which, by \eqref{f*12}, is equivalent to
\begin{eqnarray*}&& T_{\overline{f_2}}T_{f_1}w-\sum_{k=1}^n\la h,u_k\ra S^*v_k+\sum_{i=1}^n\la \frac{h}{z}+w,u_i\ra v_i=0.\end{eqnarray*}
Now choosing
\begin{eqnarray*} w=\sum_{k=1}^n\la h,u_k\ra  T_{f_1^{-1}}T_{\overline{f_2}^{-1}}(S^*v_k)\end{eqnarray*} and using \eqref{Tf*12}, the  result follows. Hence we can present a theorem on the nearly $S^*$-invariant $\Ker R_n$ with finite defect.

\begin{thm}\label{thm fj} Suppose $g=f_1\overline{f_2}$ with $f_j\in \mathcal{G}H^\infty$ for $j=1,2$. Then the subspace $\Ker R_n$ is nearly $S^*$-invariant with defect at most $n$ and defect space $$F=\bigvee \{ T_{f_1^{-1}} T_{\overline{f_2}^{-1}} (S^*v_i),\; i\in A_n\}.$$ \end{thm}
The following is a remark on two special cases of Theorem \ref{thm fj}.
\begin{rem} $(i)$\; For the operator $R_n$ in \eqref{RnT} with $\overline{g}\in \mathcal{G}H^\infty$, $\Ker R_n$ is nearly $S^*$-invariant with defect at most $n$ and defect space $$F=\bigvee\{ T_{g^{-1}}( S^*v_i),\;i\in A_n  \}.$$
$(ii)$\;For the operator $R_n$ in \eqref{RnT} with $g\in \mathcal{G}H^\infty,$ $\Ker R_n$ is nearly $S^*$-invariant with defect at most $n$ and defect space $$F=\bigvee \{ T_{g^{-1}}(S^* v_i),\; i\in A_n\}.$$ \end{rem}
 \subsection{$g(z)=\overline{\theta(z)}$ with $\theta$ a nonconstant inner function}

In this case, $T_{\overline{\theta}}$ is a special conjugate analytic Toeplitz operator with kernel $K_\theta$. And then the relation \eqref{S*} becomes
 \begin{eqnarray*}
\psi:=\overline{\theta}\frac{h}{z} +\sum_{k=1}^n\la h,u_k\ra S^*v_k \in \overline{H_0^2},
 \end{eqnarray*} with
 \begin{eqnarray}
 \theta\psi=\frac{h}{z} +\sum_{k=1}^n\la h,u_k\ra \theta S^*v_k \in H^2.\label{S*22}
 \end{eqnarray}
  \vspace{1mm}
The desired relation \eqref{questionw} now takes the form
 \begin{eqnarray}
  \overline{\theta}(\frac{h}{z}+w)+\sum_{i=1}^n \la\frac{h}{z}+w,u_i\ra v_i\in \overline{H_0^2},\end{eqnarray}   which, by \eqref{S*22}, is equivalent to
  \begin{eqnarray}
  &&\psi-\sum_{k=1}^n\la h,u_k\ra S^*v_k+\overline{\theta}w \nonumber\\&&\quad+\sum_{i=1}^n\la \psi-\sum_{k=1}^n\la h,u_k\ra S^*v_k+\overline{\theta}w,\overline{\theta}u_i\ra v_i\in \overline{H_0^2}.\quad\quad\label{psi question}\end{eqnarray}
We denote the decompositions of $u_i$ and $\psi$ as below: $u_i=u_{i1}+\theta u_{i2}$ with $u_{i1}=P_{K_\theta}u_i\in K_\theta,\;u_{i2}\in H^2,$  and $\psi=\psi_1+\psi_2$ with $\psi_1\in B:=\bigvee\{\overline{\theta}u_{i1},
\;i\in A_n\}\subset\overline{H_0^2}$ and $\psi_2\in \overline{H_0^2}\ominus B.$ So it is clear that
\begin{eqnarray*}\la \psi_2, \overline{\theta} u_i\ra=0\;\mbox{and} \; \la \psi_2, u_{i2}\ra=0\;\mbox{for all}\;i\in A_n.\end{eqnarray*}
The above indicates that \eqref{psi question} is equivalent to
\begin{eqnarray*}&&(\psi_1-\sum_{k=1}^n\la h,u_k\ra S^*v_k+\overline{\theta}w)\nonumber \\&&\quad+\sum_{i=1}^n\la(\psi_1-\sum_{k=1}^n
\la h,u_k\ra S^*v_k+\overline{\theta}w),\overline{\theta}u_{i1}+u_{i2}\ra v_i \in \overline{H_0^2}. \quad\end{eqnarray*}
Choosing \begin{eqnarray*} w=\sum_{k=1}^n\la h,u_k\ra \theta S^*v_k-\theta\psi_1,\end{eqnarray*} the above desired relation is true and the defect space $F$ is
\begin{eqnarray*} F=\bigvee\{\theta S^*v_i, P_{K_\theta}u_i, i\in A_n\}=\bigvee\{\theta S^*v_i, P_{K_\theta}u_k,  i\in A_n,\;k\in \Lambda\},\end{eqnarray*} where $\Lambda$ denotes the subset of $A_n$ consisting of all $k\in A_n$ such that $P_{K_\theta}u_k\neq 0$, i.e. $\theta\nmid u_k$. So in conclusion we have the following theorem.

\begin{thm}\label{thm oltheta}Suppose $g(z)=\overline{\theta(z)}$ with $\theta$ an inner function. Then the subspace $\Ker R_n$ is nearly $S^*$-invariant with defect at most $n+ |\Lambda|$ and defect space $$F=\bigvee \{\theta S^* v_i, \; P_{K_\theta}u_k,\;i\in A_n,\;k\in \Lambda\},$$ with $\Lambda\subset A_n$ consisting of all $k\in A_n$ such that $\theta\nmid u_k$.
 \end{thm}

\section{The application of the C-G-P theorem}\label{sec:4}
In this section, we apply a recent theorem (for short the C-G-P Theorem) by Chalendar--Gallardo--Partington to represent the kernels of rank one perturbations of Toeplitz operators in terms of backward shift-invariant subspaces. We shall take $n=1$ and denote the operator \begin{eqnarray*}R_1 f=T_g f +\la f, u\ra v \end{eqnarray*} with $\|u\|=1$ and $S^* v\neq 0.$  First we cite the C-G-P Theorem on nearly $S^*$-invariant subspaces with defect $m$ from \cite{CGP}.

\begin{thm}\cite[Theorem 3.2]{CGP}\label{CGP}\;Let $M$ be a closed subspace that is nearly $S^*$-invariant with defect $m.$ Then

$(1)$ \;in the case where there are functions in $M$ that do not vanish at $0,$ then
$$M=\{f: \;f(z)=k_0(z)f_0(z)+z\sum_{j=1}^mk_j(z)e_j(z):\;(k_0,\cdots,k_m)\in K\},$$ where $f_0$ is the normalized reproducing kernel for $M$ at $0,$\;$\{e_1,\cdots, e_m\}$ is any orthonormal basis for the defect space $F$, and $K$ is a closed $S^*\oplus \cdots\oplus S^*$-invariant subspace of the vector-valued Hardy space $H^2(\mathbb{D}, \mathbb{C}^{m+1}),$ and $\|f\|^2=\sum_{j=0}^m\|k_j\|^2.$

$(2)$\;in the case where all functions in $M$ vanish at $0,$ then
$$M=\{f: \;f(z)=z\sum_{j=1}^mk_j(z)e_j(z):\;(k_1,\cdots,k_m)\in K\},$$ with the same notation as in $(1),$ except that $K$ is now a closed $S^*\oplus \cdots\oplus S^*$-invariant subspace of the vector-valued Hardy space $H^2(\mathbb{D}, \mathbb{C}^{m}),$ and $\|f\|^2=\sum_{j=1}^m\|k_j\|^2.$\end{thm}

The following proposition asserts that the kernels of some Toeplitz operators with special symbols are model spaces.

\begin{prop}\cite[Proposition 5.8]{GMR} \label{Hinfinity}\;Let $\varphi\in H^\infty\setminus\{0\}$ and let $\eta$ be the inner factor of $\varphi,$ then $$\Ker T_{\overline{\varphi}}=K_\eta.$$\end{prop}
Now we apply the C-G-P Theorem to represent $\Ker R_1$ by backward shift-invariant subspaces in several important cases. Note that we can find $K$ as the largest $S^*$-invariant subspace such that  $$S^{*n} k_0f_0+z\sum_{j=1}^m S^{*n} k_j e_j\in M \quad \hbox{or} \quad z\sum_{j=1}^m S^{*n} k_j e_j\in M$$
for all $n\in \mathbb{N}.$
\subsection{$g=0$ a.e. on $\mathbb{T}$} In this case $M=\Ker R_1=H^2\ominus \bigvee\{u\},$ which is a vector hyperplane. It is clear that such a hyperplane is the solution of a single linear equation. Also Proposition \ref{prop g=0} showed that $\Ker R_1$ is nearly $S^*$-invariant with a $1$-dimensional defect space $F=\bigvee\{ u\}.$ Using Theorem \ref{CGP}, we deduce a corollary on the representations of $\Ker R_1$.

\begin{cor}\label{cor g=0} Given a nearly $S^*$-invariant subspace $M=H^2\ominus \bigvee\{u\}$ with defect $1$, let $f_0=P_M 1=1-\overline{u(0)} u,$ $v_0=P(u-u(0)|u|^2)$ and $v_1=P(\overline{z}|u|^2).$ Then

$(1)$ in the case $P_M 1\neq 0,$ we have
$$M=\{f:\;f=k_0f_0 +k_1zu:\; (k_0, k_1)\in K\},$$ with an $S^*\oplus S^*$-invariant subspace $K=\{(k_0,k_1): \la k_0,z^n v_0\ra+\la k_1, z^n v_1\ra=0 \;\mbox{for}\;n\in \mathbb{N}\}$.

$(2)$ in the case $P_M 1= 0,$ we have $$M=\{f:\;f=k_1zu:\;  k_1\in K\},$$ with an $S^*$-invariant subspace $K=\{ k_1:\;\la  k_1, z^n v_1\ra=0 \;\mbox{for}\;n\in \mathbb{N}\}$.
\end{cor}
Here we show some examples illustrating the variety of  subspaces $K$ that can occur.\vspace{1.5mm}

\begin{exm}~$(i)$\;Suppose $u=1$, then $ M=zH^2$, $f_0=0$ and $v_0=v_1=0.$ So Corollary \ref{cor g=0} $(2)$ implies $M$ has the representation
  \begin{eqnarray*}M=\{f:\; f= zk_1:\;k_1\in K\} \end{eqnarray*} with $K=H^2$ a trivial $S^*$-invariant subspace.\vspace{1mm}

$(ii)$\;Suppose $u$ is a nonconstant inner function, then $M=K_u\oplus zuH^2$, $f_0=1- \overline{u(0)}u\neq 0$ and $v_0=u-u(0)$, $v_1=0.$ So Corollary \ref{cor g=0} $(1)$ implies $M$ has the representation
\begin{eqnarray*} M=\{f:\; f=k_0(1- \overline{u(0)}u)+k_1zu:\;(k_0,k_1)\in K\} \end{eqnarray*} with an $S^*\oplus S^*$-invariant subspace $ K=K_\eta\times H^2$, where $\eta$ is the inner factor of $ v_0$. Besides, Proposition \ref{Hinfinity} is used to show that $K$ is backward shift-invariant. \vspace{1.0mm}

$(iii)$\;Suppose $u$ is a normalized reproducing kernel of $H^2,$ that is $u(z)=\sqrt{1-|\alpha|^2}(
1-\overline{\alpha}z)^{-1},\;\;\alpha\in \mathbb{D}\setminus\{0\},$  then $M=\{f:\; f(\alpha)=0\}$, $f_0=\overline{\alpha}(\alpha-z)(1-\overline{\alpha}z)^{-1}\neq 0$ and $v_0=0,\;v_1= \overline{\alpha}(1-\overline{\alpha}z)^{-1}$. So Corollary \ref{cor g=0} $(1)$ implies $M$ has the representation
\begin{eqnarray*}M=
\{f:\;f=\overline{\alpha}k_0\frac{\alpha-z}{1-\overline{\alpha}z}+
zk_1\frac{\sqrt{1-|\alpha|^2}}{1-\overline{\alpha}z}:\;(k_0,k_1)\in K\},\end{eqnarray*}
with an $S^*\oplus S^*$-invariant subspace $K=H^2 \times \{0\}$. \vspace{1mm}

$(iv)$ \;Suppose $u(z)=(1+z^k)/\sqrt{2}$ with $k\geq 1$, then $M=\bigvee\{1-z^k, z, \cdots, z^{k-1}, z^{k+1}, z^{k+2},\cdots\},$   $f_0=2^{-1}(1-z^k)\neq 0$ and $v_0=z^k/(2\sqrt{2})$, $v_1=2^{-1}z^{k-1}.$  So Corollary \ref{cor g=0} $(1)$ implies $M$ has the representation
\begin{eqnarray*}M=\{f:\;f=k_0\frac{1-z^k}{2}+zk_1 \frac{1+z^k}{\sqrt{2}}:
\;(k_0,k_1)\in K\} \end{eqnarray*} with an $S^*\oplus S^*$-invariant subspace $K=\{(k_0, k_1): \sqrt{2}(S^{*})^{k-1} k_1=-(S^*)^{k}k_0,\;k_0\in H^2\}$.
\end{exm}

\subsection{$g=\theta$ an inner function} In this case,
$M=\Ker R_1\subset \bigvee\{ \overline{\theta}v\}$. Take
any vector $f=\lambda\overline{\theta} v \in M $ satisfying $R_1 f=0,$ which is equivalent to $\lambda (1+\la \overline{\theta} v, u\ra)=0.$ If $1+\la \overline{\theta}v,u\ra \neq 0$, then $\lambda =0,$ meaning $M=\{0\}$ a trivial $S^*$-invariant subspace. 
So suppose $1+\la \overline{\theta} v, u\ra=0,$ and then $M=\bigvee\{\overline{\theta} v\}$, which is nearly $S^*$-invariant with a $1$-dimensional defect space $F=\bigvee\{ S^*({\overline{\theta}}v)\}$ from Theorem \ref{thm theta1}. 
So Theorem \ref{CGP} implies a corollary on the representations of $\Ker R_1.$
\begin{cor} Given a nearly $S^*$-invariant subspace $M=\bigvee\{\overline{\theta} v\}$ with defect $1$, then

$(1)$ in the case $a_0:=\la \overline{\theta} v, 1\ra\neq 0,$
let $f_0=P_{M}1=\overline{a_0}\|v\|^{-2}\overline{\theta} v,$ we have
 \begin{eqnarray*} M=\{f:\;f= k_0f_0:\;(k_0,0)\in K\},
\end{eqnarray*}  with an $S^*\oplus S^*$-invariant subspace $K=\mathbb{C}\times \{0\}$.\vspace{0.1mm}

$(2)$ in the case $a_0:=\la \overline{\theta} v, 1\ra=0$, we have
\begin{eqnarray*}M =\{f:\;f=\|S^*({\overline{\theta}}v)\|^{-1} k_1 \overline{\theta} v:\;k_1\in K  \}, \end{eqnarray*} with an $S^*$-invariant subspace $K=\mathbb{C}$. \end{cor}
\begin{proof} $(1)$ in this case, using Theorem \ref{CGP} $(1)$, we represent $M$ by \begin{eqnarray*}M &=&\{f:\;f= k_0f_0+ zk_1 \frac{S^*({\overline{\theta}}v)}
{\|S^*({\overline{\theta}}v)\|}:\;(k_0,k_1)\in K\}
\nonumber\\&=&
\{f:\;f=\frac{\overline{a_0}}{\|v\|^2}k_0\overline{\theta} v + \|S^*({\overline{\theta}}v)\|^{-1} k_1  (\overline{{\theta}}v -a_0):\;(k_0, k_1)\in K\}.
\end{eqnarray*} Since $M=\bigvee\{\overline{\theta} v\}$, it yields
\begin{eqnarray*}k_0\in \mathbb{C}\; \;\mbox{and}\; \;\|S^*({\overline{\theta}}v)\|^{-1} k_1 ({\overline{\theta}}v -a_0) =\mu \overline{\theta} v\;\mbox{with $\mu\in \mathbb{C}$},\end{eqnarray*}
which is equivalent to $k_0\in \mathbb{C}$ and $k_1=0$ due to $a_0\neq 0.$  So the statement $(1)$ is true. The statement $(2)$ can be similarly shown by Theorem \ref{CGP} $(2)$.\end{proof}
\subsection{$g=f_1\overline{f_2}$ with $f_j\in \mathcal{G}H^\infty$ for $j=1,2$.}  In this case, $M=\Ker R_1\subset \bigvee\{ f_1^{-1}(T_{\overline{f_2}^{-1}}v) \}.$ Take any vector $f=\lambda f_1^{-1}(T_{\overline{f_2}^{-1}}v) \in M$ such that $R_1 f=0,$ which is equivalent to $\lambda (1+\la f_1^{-1}(T_{\overline{f_2}^{-1}}v), u\ra)=0.$ It is clear $M=\{0\}$ is a trivial $S^*$-invariant subspace for $1+\la f_1^{-1}T_{\overline{f_2}^{-1}}v, u\ra \neq 0$. Now we always assume $1+\la f_1^{-1}(T_{\overline{f_2}^{-1}}v), u\ra=0,$ and obtain   $M=\bigvee\{ f_1^{-1}(T_{\overline{f_2}^{-1}}v)\}$, which is nearly $S^*$-invariant with a $1$-dimensional defect space $F=\bigvee\{ f_1^{-1}T_{\overline{f_2}^{-1}}(S^*v )\}$ from Theorem \ref{thm fj}. Denote the Taylor coefficients of $T_{\overline{f_2}^{-1}}v$ and $f_1^{-1}$ by $\{a_k\}_{k\in \mathbb{N}}$ and $ \{b_k\}_{k\in \mathbb{N}}$, respectively. So $\la f_1^{-1}T_{\overline{f_2}^{-1}}v, 1 \ra=a_0b_0$, and using Theorem \ref{CGP}, we deduce a corollary on the representations of $\Ker R_1$.

\begin{cor} Given a nearly $S^*$-invariant subspace\\  $M=\bigvee\{f_1^{-1}(T_{\overline{f_2}^{-1}}v)\}$ with defect $1$, then

$(1)$ in the case $a_0b_0\neq 0,$ let $f_0=P_M1=\overline{a_0b_0} \|f_1^{-1} T_{\overline{f_2}^{-1}}v\|^{-2} f_1^{-1}T_{\overline{f_2}^{-1}}v$; then we have
\begin{eqnarray*}
M=\{f:\;f=k_0f_0:\;(k_0,0)\in K\},\end{eqnarray*}
with an $S^*\oplus S^*$-invariant subspace $K=\mathbb{C}\times \{0\}$.

$(2)$ in the case $a_0b_0=0,$ we have \begin{eqnarray*}M=\{f:\;f= k_1\frac{ f_1^{-1} T_{\overline{f_2}^{-1}} v }{\|  f_1^{-1}T_{\overline{f_2}^{-1}}(S^*v)\|}:\;k_1\in K\}\end{eqnarray*} with  $K=\mathbb{C}$ an $S^*$-invariant subspace.\end{cor}
\begin{proof}$(1)$ in this case,   Theorem \ref{CGP} $(1)$ gives
\begin{eqnarray*}
M=\{f:\;f=k_0 f_0+ k_1 \frac{f_1^{-1}(T_{\overline{f_2}^{-1}}v-a_0)}
{{\|f_1^{-1}T_{\overline{f_2}^{-1}}(S^*v) \|}}:\;(k_0, k_1)\in K\}, \end{eqnarray*} due to $zf_1^{-1} T_{\overline{f_2}^{-1}}(S^*v)]=f_1^{-1}z[S^* (T_{\overline{f_2}^{-1}}v)]=f_1^{-1}(T_{\overline{f_2}^{-1}}v-a_0).$ Further by $M=\bigvee\{f_1^{-1}(T_{\overline{f_2}^{-1}}v)\}$, it follows that
\begin{eqnarray*}k_0\in \mathbb{C}\;\; \mbox{and} \;\;k_1\frac{f_1^{-1}(T_{\overline{f_2}^{-1}}
v-a_0)}{\|f_1^{-1}T_{\overline{f_2}^{-1}}(S^*v) \|} =\mu f_1^{-1}(T_{\overline{f_2}^{-1}}v) \;\mbox{with $\mu\in \mathbb{C},$}\end{eqnarray*} which is equivalent to
$k_0\in \mathbb{C}\;\mbox{and}\;k_1=0$ by $a_0\neq 0.$

$(2)$ in this case, it follows either $a_0=0$ or $b_0=0$ and $f_0=P_M1=0.$ If $b_0=0$, Theorem \ref{CGP} $(2)$ implies that
\begin{eqnarray*}M=\{f:\;f=k_1 \frac{f_1^{-1}(T_{\overline{f_2}^{-1}}v-a_0)}{\| f_1^{-1} T_{\overline{f_2}^{-1}}(S^*v)\|}:\;k_1\in K\},\end{eqnarray*}
 which is valid if and only if $a_0=0$  and $k_1\in \mathbb{C}.$ \end{proof}
\subsection{ $g=\overline{\theta}$ with $\theta$ nonconstant inner function}
Because of its link with model spaces, this case is of particular interest. For every $h\in \Ker R_1,$ the equation \eqref{Tg1} is equivalent to
$ h+\la h, u\ra \theta v\in \theta\overline{H_0^2}.$ So \begin{eqnarray*}M=\Ker R_1\subset (H^2 \cap  \theta\overline{H_0^2}) \oplus \bigvee\{ \theta v\}=K_{\theta}\oplus \bigvee\{ \theta v\}.\end{eqnarray*}
Take any vector $h= h_1+\lambda \theta v\in M$ with $h_1\in K_\theta\;\mbox{and}\;\lambda \in \mathbb{C}$, such that $R_1 h=0,$ which is equivalent to
\begin{eqnarray} \lambda(1 + \la \theta v, u\ra)=-\la h_1, u \ra .\label{Mequation} \end{eqnarray}
Now we divide this into two subsections to represent $M=\Ker R_1$ in terms of backward shift-invariant subspaces.

\subsubsection{$\theta|u$.} In this case, the equation \eqref{Mequation} now is changed into $\lambda(1+\la \theta v, u\ra)=0$. If $1+\la \theta v, u\ra\neq 0,$ then $\lambda=0$ and $M=K_\theta$ a nearly $S^*$-invariant subspace. So we suppose $1+\la \theta v, u\ra =0$, and then $M=K_\theta\oplus\bigvee\{\theta v\}$ is nearly $S^*$-invariant with a $1$-dimensional defect space $F=\bigvee\{ \theta S^* v\}$ from Theorem \ref{thm oltheta}. Using Theorem \ref{CGP}, we obtain  a corollary on the representation of $\Ker R_1.$
\begin{cor} Given a nearly $S^*$-invariant subspace $M=K_\theta\oplus\bigvee\{\theta v\}$ with defect $1$, and let $f_0=P_M 1=1-\overline{\theta(0)}\theta +  \overline{\theta(0)}\overline{v(0)} \|v\|^{-2}  \theta v,$ we have
\begin{eqnarray}M &=&\{f:\; f=k_0- (k_0\overline{\theta(0)}+ k_1\frac{v(0)}{\| S^*v\|} )\theta+ ( k_0\frac{\overline{\theta(0)v(0)}}{\|v\|^2} \nonumber\\&&\quad\quad\quad \quad+ \| S^* v\|^{-1} k_1)\theta v:\;(k_0,k_1)\in K\},\label{Mtheta}\end{eqnarray} with an $S^*\oplus S^*$-invariant subspace
$K=\{(k_0, k_1):\;k_i \;\mbox{satisfies}\;\eqref{thetav1}\;\mbox{for}\;i=0,1\},$
where \begin{eqnarray}
           k_0-(k_0\overline{\theta(0)}+ k_1\frac{v(0)}{\| S^* v\|})\theta\in K_\theta\;\;\mbox{and}\;\;
 k_0\frac{\overline{\theta(0)v(0)}}{\|v\|^2} + \| S^* v\|^{-1}k_1\in \mathbb{C}.\quad \label{thetav1}             \end{eqnarray}
\end{cor}
\begin{proof} By Theorem \ref{CGP} $(1)$, we obtain
\begin{eqnarray*}M=\{f:\; f=k_0f_0 + k_1 \frac{\theta (v-v(0)) }{\| S^* v\|}:\;(k_0,k_1)\in K\},\end{eqnarray*} which equals $K_\theta\oplus\bigvee\{\theta v\}$ implying the desired representation in \eqref{Mtheta}. It is clear the second relation in \eqref{thetav1} holds for $S^*k_i$, $i=1,2$. At the same time, the first relation in \eqref{thetav1} together with the fact $K_\theta$ is an $S^*$-invariant subspace verify that
\begin{eqnarray*}
 Y_\theta:=S^*k_0-S^*(k_0\overline{\theta(0)}\theta+ \frac{ v(0)k_1}{\| S^* v\|}\theta) \in K_\theta.
  \end{eqnarray*}
Then it turns out that
\begin{eqnarray*}&&S^*k_0-(S^* k_0\overline{\theta(0)}+\frac{v(0)}{\|S^* v\|} S^*k_1)\theta\\&=&Y_\theta +\overline{\theta(0)}k_0(0)\frac{ \theta-\theta(0)}{z}+\frac{v(0)k_1(0)}{\|S^*v\|}\frac{\theta-\theta(0)}{z}
\\&=&Y_\theta +(\overline{\theta(0)}k_0(0)+\frac{v(0)k_1(0)}{\|S^*v\|}) T_{\overline{z}} \theta \in K_\theta,\end{eqnarray*} since $\la T_{\overline{z}} \theta, \theta \ra=\la 1, z \ra=0 $ holds. This means the first relation in \eqref{thetav1} also holds for $S^*k_i$, $i=1,2$. So $K$ is an $S^*\oplus S^*$-invariant subspace. \end{proof}

\subsubsection{ $\theta\nmid u$.} In this case, we decompose $u$ into $u=u_1+u_\theta$ with nonzero $u_1\in K_\theta$ and $u_\theta\in \theta H^2.$ Then the identity \eqref{Mequation} becomes \begin{eqnarray}\lambda (1+\la \theta v, u_\theta \ra)=-\la h_1, u_1\ra.\label{general1}\end{eqnarray} Especially Theorem \ref{thm oltheta} implies $\Ker R_1$ is nearly $S^*$-invariant with a $2$-dimensional defect space $F=\bigvee \{\theta S^* v, u_1\}.$ For later use we present a remark concerning the projection $P_M1$.

\begin{rem}\label{remprojection} Let $M=\Ker R_1\subset N:=K_\theta\oplus\bigvee\{\theta v\}$, and denote $N=M\oplus \bigvee\{G\}$ with $G=g+\mu \theta v$, where $g\in K_\theta$ and $\mu \in \mathbb{C}$. Then
\begin{eqnarray}  P_M 1 &=&1-\overline{\theta(0)}\theta +\frac{\overline{\theta(0)v(0)}}{\|v\|^2}\theta v\nonumber\\&&-\frac{\la 1-\overline{\theta(0)}\theta, g\ra+\overline{\theta(0)v(0)\mu}}{\|g\|^2+|\mu|^2\|v\|^2}(g+\mu \theta v).\label{one}\end{eqnarray}
\end{rem}

For simplicity, we denote $w_{\theta}:= 1+\la \theta v, u_\theta \ra$ and $$\rho_\theta:=\frac{ \overline{u_1(0)}+ \overline{\theta(0)v(0)} w_\theta}{\|u_1\|^2+|w_\theta|^2\|v\|^2}.$$
Applying Theorem \ref{CGP}, we present a corollary on $\Ker R_1.$

\begin{cor}\label{cor oltheta2}
$(1)$ In the case $w_{\theta}\neq 0,$
$M=N\ominus\bigvee\{u_1+\overline{w_\theta}\theta v\}$ is nearly $S^*$-invariant with defect $2$, and letting
 \begin{eqnarray}&&f_0=P_M1= 1-\overline{\theta(0)}\theta+ \overline{\theta(0)v(0)} \|v\|^{-2}  \theta v-\rho_\theta(u_1+\overline{w_\theta} \theta v),\quad\label{f02}
  \\&& v_0= P(u_1+\overline{w_\theta}\theta v-\theta(0)\overline{w_\theta}v+\theta(0)v(0)\|v\|^{-2}
\overline{w_\theta}|v|^2-\overline{\rho_\theta}|u_1+
\overline{w_\theta}\theta v|^2),\nonumber\\&& v_1=\|S^*v\|^{-1}P(\overline{w_\theta}v(\overline{v-v(0)}))\;
\mbox{and}\; v_2 =\|u_1\|^{-1}P(\overline{z} |u_1|^2),\nonumber\end{eqnarray} we have
\begin{eqnarray}M &=&\{f: f=k_0- (k_0\overline{\theta(0)}+k_1\frac{ v(0)}{\|S^* v\|})\theta+( k_0\frac{\overline{\theta(0)v(0)}}{\|v\|^2}+\frac{k_1}{\|S^* v\|}-k_2\frac{z\overline{w_\theta}}{\|u_1\|})\theta v\nonumber \\&&\quad
+(-k_0\rho_\theta+k_2\frac{z}{\|u_1\|})(u_1+\overline{w_\theta} \theta v) :\;(k_0, k_1, k_2)\in K\},\quad \quad\quad\label{M2}
\end{eqnarray}
with an $S^*\oplus S^* \oplus S^*$-invariant subspace $K=\{(k_0, k_1, k_2):\; k_i$\;$\mbox{satisfies}$ \; $ \eqref{Akg}\;\mbox{for}\;i=0,1,2\;\},$ where
\begin{eqnarray}\left\{
                   \begin{array}{ll}
  k_0-(k_0\overline{\theta(0)}
+k_1\frac{ v(0)}{\|S^* v\|})\theta\in K_\theta, \vspace{1.5mm}\\
                      k_0 \frac{\overline{\theta(0)v(0)}}{\|v\|^2}
                      +\frac{k_1}{\|S^* v\|}-k_2\frac{z\overline{w_\theta}}{\|u_1\|}  \in \mathbb{C},\vspace{2.5mm}\\
                    \la k_0,\; z^n v_0\ra +\la k_1, \; z^n v_1 \ra +\la k_2, \;z^nv_2\ra=0\;\;
\mbox{for all $n\in \mathbb{N}.$}
                   \end{array}
                 \right.\label{Akg}
 \end{eqnarray}
 
 $(2)$ In the case $w_{\theta}=0$, $M=N \ominus\bigvee\{u_1\}$ is nearly $S^*$-invariant with defect $2$, and letting 
\begin{eqnarray*}&&f_0=P_{M}1=1-\overline{\theta(0)}\theta -\overline{u_1(0)}\|u_1\|^{-2}u_1+ \overline{\theta(0) v(0)}\|v\|^{-2} \theta v,\quad\quad\label{f01}\\&&v_0=P(u_1- u_1(0)\|u_1\|^{-2} |u_1|^2) \quad \mbox{and}\quad v_2=P( \|u_1\|^{-1}\overline{z}|u_1|^2 ),\nonumber\end{eqnarray*} we have
\begin{eqnarray*} M&=&\{f: f=k_0- (k_0\overline{\theta(0)} +k_1\frac{v(0)}{\|S^* v\|})\theta+ (k_0\frac{\overline{\theta(0)v(0)}}{\|v\|^2}+ \frac{k_1}{\|S^* v\|})\theta v\nonumber\\&&\quad\quad\quad +(-k_0\frac{\overline{u_1(0)}}{\|u_1\|^2}+
k_2\frac{z }{\|u_1\|})u_1:\;(k_0, k_1, k_2)\in K\}, \end{eqnarray*}
with an $S^*\oplus S^* \oplus S^*$-invariant subspace $K=\{(k_0, k_1, k_2):\; k_i\;\mbox{satisfies} \; \eqref{u31}$ $\;\mbox{for}\;i=0,1,2\;\},$ where
 \begin{eqnarray}\left\{
                   \begin{array}{ll}
                     k_0-(k_0\overline{\theta(0)} +k_1\frac{v(0)}{\|S^*v\|})\theta \in K_\theta,\vspace{1.5mm} \\
                    k_0\frac{\overline{\theta(0)v(0)}}{\|v\|^2}+
                    \|S^* v\|^{-1}k_1\in \mathbb{C},  \vspace{2.5mm}\\
                    \la k_0,z^nv_0\ra+\la k_2, z^n v_2\ra =0\;\;\mbox{for}\;n\in \mathbb{N}.\label{u31}
                   \end{array}
                 \right.
   \end{eqnarray}
   
\end{cor}

\begin{proof}For the case $w_{\theta}\neq 0,$ the equation \eqref{general1} implies  $\lambda=- w_{\theta}^{-1}\la h_1, u_1\ra  $ and then
$M=\Ker R_1=\{f:\;f= k-w_\theta^{-1}\la k,u_1\ra\theta v,\;k\in K_\theta\}.$ By some calculations, it follows \begin{eqnarray}M=N\ominus \bigvee\{u_1+\overline{w_\theta}\theta v\}.\label{MN}\end{eqnarray} Letting $g=u_1$ and $\mu=\overline{w_\theta}$ in \eqref{one}, we obtain $f_0$ in \eqref{f02}. By Theorem \ref{CGP} $(1)$, it follows
\begin{eqnarray*}M=
\{f:\; f=k_0f_0+k_1\frac{\theta(v-v(0))}{\|S^* v\|} +k_2\frac{zu_1}{\|u_1\|}:\;(k_0, k_1, k_2)\in K\}
\end{eqnarray*} which together with \eqref{MN} imply the representation of $M$ in \eqref{M2}. Note the third formula in \eqref{Akg} holds for $S^*k_i$, $i=0, 1, 2.$ Following the similar lines for proving \eqref{thetav1} is $S^*$-invariant, it is easy to check the first two relations of \eqref{Akg} are valid for $S^*k_i$, $i=0,1,2$. So $K$ is an $S^*\oplus S^*\oplus S^*$-invariant subspace. In particular, if $w_{\theta}=0$, the equation \eqref{general1} implies \begin{eqnarray*}M=N\ominus\bigvee\{ u_1\},\end{eqnarray*} which is a special case of \eqref{MN} with $w_\theta=0$. Hence we can prove the statement $(2)$ from the similar proof of statement $(1)$ with $w_\theta=0$.  \end{proof}
In order to help understand the case $g=\overline{\theta}$ with $\theta$ an inner function, we present an example for Corollary \ref{cor oltheta2}.

 \begin{exm}Let $\theta=z^m\; (m\geq 1)$ and $u=u_1+u_2$ with $u_1=z^{m-1}/4$ and $u_2\in z^m H^2.$ It easy to check the kernel of $R_1$ is $$M=\bigvee \{1, z, \cdots, z^{m-1}\ra\oplus \bigvee\{ z^m v\}\ominus\bigvee\{ z^{m-1}/4+\overline{w_\theta}z^m v\},$$ which is nearly $S^*$-invariant with $2$-dimensional defect space\\ $F=\bigvee\{z^m S^*v, z^{m-1}\}$ from Theorem \ref{thm oltheta}. If  $w_\theta=1+\la z^m v, u_\theta\ra\neq 0,$ then Corollary \ref{cor oltheta2} $(2)$ indicates the following representation for $M$:
  \begin{eqnarray*} M&=& \{f:\;f=k_0
- k_1\frac{v(0)}{\|S^* v\|} z^m +(k_1 \|S^* v\|^{-1} -4k_2z\overline{w_\theta})z^m v\quad\quad\quad\\&&+k_2z(z^{m-1}+4\overline{w_\theta}z^m v):\;(k_0, k_1, k_2)\in K\},\end{eqnarray*} with an $S^*\oplus S^*\oplus S^*$-invariant subspace $K=\{(k_0, k_1, k_2):\; k_i\; \mbox{satisfies}$ $\eqref{Akgm} \;\mbox{for} \;i=0,1,2\}$ where
\begin{eqnarray}  \left\{
                    \begin{array}{ll}
                 k_0-k_1\frac{ v(0)}{\|S^* v\|} z^m\in K_{z^m},\vspace{1.5mm}\\
                 k_1 \|S^* v\|^{-1}- 4k_2z\overline{w_\theta}  \in \mathbb{C},\vspace{1.5mm}\\
                    \la k_0,\; z^n v_0\ra +\la k_1, \; z^n v_1\ra =0\;\mbox{for}\;n\in \mathbb{N},
                    \end{array}
                  \right.
                   \label{Akgm}\end{eqnarray}
with \begin{eqnarray*} v_0=4^{-1}z^{m-1}+\overline{w_\theta}z^m v,\;\; v_1=\|S^*v\|^{-1} P(\overline{w_\theta}v(\overline{v-v(0)})). \end{eqnarray*}
\end{exm}

{\bf Acknowledgments}. The authors thank the referee for many useful comments which improve the presentation considerably. This work was done while the first author was visiting the University of Leeds. She is grateful to the School of Mathematics at the University of Leeds for its hospitality. The first author is supported  by the National Natural Science Foundation of
China (Grant No. 11701422).

\end{document}